\documentclass[11pt]{article}
\usepackage{graphicx}
\usepackage{amssymb}
\usepackage{amsmath,amsthm}
\usepackage{verbatim}
\usepackage{amsfonts}
\usepackage[T1]{fontenc}
\usepackage{color}

\newtheorem{thm}{Theorem}

\newtheorem{defn}[thm]{Definition}
\newtheorem{lem}[thm]{Lemma}

\newtheorem*{thm*}{Theorem}

\newcommand{\E}{\mathcal{E}}

\newcommand{\G}{\mathcal{G}}
\newcommand{\Hy}{\mathcal{H}}

\newcommand{\List}{\mathcal{L}}

\newcommand{\La}{\mathcal{L}}

\newcommand{\bP}{\mathbb{P}}
\newcommand{\bE}{\mathbb{E}}

\newcommand{\bF}{\mathbb{F}}
\newcommand{\bN}{\mathbb{N}}
\newcommand{\V}{\mathcal{V}}

\newcommand{\mzero}{\mathbf{0}}
\newcommand{\Aut}{\mathrm{Aut}}

\def\nor#1#2{{\bf N}_{{#1}}{{(#2)}}}

\title{Vertex transitive graphs $G$ with $\chi_D(G) > \chi(G)$ and small automorphism group}
\author{Niranjan Balachandran\footnote{Department of Mathematics, Indian Institute of Technology Bombay, Mumbai, India. email: niranj@math.iitb.ac.in},  Sajith Padinhatteeri\footnote{Department of Mathematics, Indian Institute of Technology Bombay, Mumbai, India. email: sajith@math.iitb.ac.in}, and Pablo Spiga\footnote{Dipartimento Di Matematica E Applicazioni, University of Milano-Bicocca, Milano Italy, Email: pablo.spiga@unimib.it}}

\date{}
\begin{document}
\maketitle

\begin{abstract}
For a graph $G$ and a positive integer $k$, a vertex labelling $f:V(G)\to\{1,2\ldots,k\}$ is said to be $k$-distinguishing if no non-trivial automorphism of $G$ preserves the sets $f^{-1}(i)$ for each $i\in\{1,\ldots,k\}$. The distinguishing chromatic number of a graph $G$, denoted $\chi_D(G)$, is defined as the minimum $k$ such that there is a $k$-distinguishing labelling of $V(G)$ which is also a proper coloring of the vertices of $G$. In this paper, we prove the following theorem: Given $k\in\bN$, there exists an infinite sequence of vertex-transitive graphs $G_{i}=(V_i,E_i)$ such that
\begin{enumerate}
\item $\chi_D(G_i)>\chi(G_i)>k$,
\item $|\Aut(G_i)|=O_k(|V_i|)$, where $\Aut(G_i)$ denotes the full automorphism group of $G_i$.
\end{enumerate}
In particular, this answers a problem raised in \cite{BP}.
\end{abstract}


2010 AMS Classification Code: 05C15, 05D40, 20B25, 05E18. 

\section{Introduction}

Let $G$ be a graph. An automorphism of  $G$ is a permutation $\varphi$ of the vertex set $V(G)$ of $G$ such that, for any  $x,y\in V(G)$, $\varphi(x),\varphi(y)$ are adjacent if and only $x,y$ are adjacent. The automorphism group of a graph $G$, denoted by $\mathrm{Aut}(G)$, is the group of all automorphisms of $G$. A graph $G$ is said to be vertex transitive if, for any $u, v \in V(G)$, there exists  $\varphi \in \mathrm{Aut}(G)$ such that $\varphi(u) = v$. 

Given a positive integer $r$,  an $r$-coloring of $G$ is a map $f:V(G) \rightarrow \{1,2, \dots, r\}$ and the sets $f^{-1}(i)$, for $i\in\{1,2\ldots,r\}$, are  the color classes of $f$. An automorphism $\varphi \in \Aut(G)$ is said to fix a color class $C$ of $f$ if $\varphi(C) = C$, where $\varphi(C) = \{\varphi(v):v \in C\}$. A coloring of $G$, with the property that no non-trivial automorphism of $G$ fixes every color class, is called a distinguishing coloring of $G$.\\

Collins and Trenk in \cite{collins} introduced the notion of the distinguishing chromatic number of a graph $G$, which is defined as  the minimum number of colors  needed to color the vertices of $G$ so that the coloring is both proper and distinguishing. Thus, the distinguishing chromatic number of $G$ is the least integer $r$ such that the vertex set can be partitioned into sets $V_1,V_2,\ldots, V_r$ such that each $V_i$ is independent in $G$, and for every non-trivial $\varphi\in \Aut(G)$ there exists some color class $V_i$ with $\varphi(V_i)\neq V_i$. The distinguishing chromatic number of a graph $G$, denoted by $\chi_D(G)$, has been the topic of considerable interest recently (see for instance, \cite{BP,kneser,Hemanshu,Collins2009}).

One of the many questions of interest regarding the distinguishing chromatic number concerns the contrast between $\chi_D(G)$ and the cardinality of $\Aut(G)$. For instance, the Kneser graphs $K(n,r)$ have very large automorphism groups and yet, $\chi_D(K(n,r))=\chi(K(n,r))$ for $n\ge 2r+1$, and $r\ge 3$ (see \cite{kneser}). The converse question is compelling: Are there infinitely many graphs $G_{n}$ with `small' automorphism groups and satisfying $\chi_D(G_{n})>\chi(G_n)$? 

The question as posed above is not actually interesting for two reasons. First, for all even $n$, $\chi_D(C_n)>\chi(C_n)=2$ and $|\Aut(C_n)|=2n$, where $C_n$ is the cycle of length $n$. Second, if one stipulates that $G$ also has arbitrarily large chromatic number, then here is a construction for such a graph. Start with a rigid graph $G$ with a leaf vertex $x$ and having large chromatic number (one can obtain this by minor modifications to a random graph, for instance); then, blow up the leaf vertex $x$ to a new disjoint set $X$ whose neighbor in the new graph $\widetilde{G}$ is the same as the neighbor of $x$ in $G$. In fact one can arrange for $\chi_D(\widetilde{G})-\chi(\widetilde{G})$ to be as large as one desires. Furthermore, since $|\Aut(\widetilde{G})|=|X|!$, this provides examples of graphs for which the automorphism groups are relatively `small' in terms of the order of the graph.

In the example above, the fact that $\chi_D(G)$ is larger than $\chi(G)$ is accounted for by a `local' reason, and that is what makes the problem stated above not very interesting. However, if one further stipulates that the graph is vertex-transitive, then the same question is highly non-trivial. In~\cite{BP}, the first and second authors constructed families of vertex-transitive graphs with $\chi_D(G)>\chi(G)>k$ and $ \Aut(G)|=O(|V(G)|^{3/2})$, for any given $k$. In this paper, we improve upon that result:
\begin{thm}\label{qn}
	Given  $k\in\bN$, there exists an infinite family of graphs $G_{n}=(V_n,E_n)$ satisfying:
	\begin{enumerate}
		\item $\chi_D(G_n)>\chi(G_{n})>k$,
		\item $G_{n}$ is vertex transitive and $|\Aut(G_{n})|<2k|V_n|$. 
	\end{enumerate}
\end{thm}

Our family of graphs consists of Cayley graphs.  To recall the definition, let $A$ be a group and let $S$ be  an inverse-closed subset of $A$, i.e., $S=S^{-1}$, where $S^{-1}:=\{s^{-1}:s\in S\}$. The Cayley graph $\mathrm{Cay}(A,S)$ is the graph with vertex set $A$ and the vertices $u$ and $v$ are adjacent in $\mathrm{Cay}(A,S)$ if and only if $uv^{-1} \in S$. 

We start with a brief description of the graphs of our construction.  For $q$, an odd prime, let $\bF_q^n$ denote the $n$-dimensional vector space over $\bF_q$. Our graphs shall be Cayley graphs $\mathrm{Cay}(\bF_q^n, S)$ for some suitable inverse-closed set $S\subset\bF_q^n$ which is obtained by taking a union of a certain collection of lines in $\bF_q^n$ and then deleting the zero element of $\bF_q^n$. More precisely, 
let $\mathcal{H}_0:=\{(x_1,x_2,\dots,x_{n-1},0) : x_i \in \bF_q, 1\le i \le n-1\}$ and let $\mathbf{0}$ denote the element $(0,\ldots,0)\in\bF_q^n$. For each line ($1$-dimensional subspace of $\bF_q^n$) $\ell\subset\bF_q^n$ satisfying $\ell\cap\Hy_0=\{\mzero\}$, pick $\ell$ independently with probability $1/2$ to form the random set $\widetilde{S}$. Our connection set $S$ for the Cayley graph $\mathrm{Cay}(\mathbb{F}_q^n,S)$ is defined by $S:=\{v \in\mathbb{F}_q^n: v\in\ell\textrm{\ for\ some\ }\ell\in\widetilde{S}\}\setminus\{\mzero\}$. Our main theorem states that with high probability, $G_{n,S}:= \mathrm{Cay}(\bF_q^n, S)$ satisfies the conditions of Theorem~\ref{qn}.

To show that these graphs have `small' automorphism groups, we prove a stronger version of Theorem 4.3 of \cite{dobson} in this particular context, which is also a result of independent interest.  

%

\begin{thm}\label{autcayley}Let $q$ be a prime power, let $n$ be a positive integer with $n\ge 2$ and let $G$ be the additive group of the $n$-dimensional vector space $\mathbb{F}_q^n$ over the finite field $\mathbb{F}_q$ of cardinality $q$, and let $\mathbb{F}_q^{*}:=\mathbb{F}_q\setminus\{\mzero\}$ be the multiplicative group of the field $\mathbb{F}_q$ with its natural group action on $G$ by scalar multiplication, and write  $K:=\mathbb{F}_q^n\rtimes \mathbb{F}_q^*$. If $S$ is a subset of $G$ with $K \le \mathrm{Aut}(\mathrm{Cay}(G,S))$, then either
	\begin{description}
		\item[(i)]$\mathrm{Aut}(\mathrm{Cay}(G,S))=K$, or
		\item[(ii)]there exists $\varphi\in \mathrm{Aut}(\mathrm{Cay}(G,S))\setminus K$ with $\varphi$ normalizing $G$. 
	\end{description}
\end{thm}

The rest of the paper is organized as follows. We start with some preliminaries in Section \ref{sec2} and then include the proofs of Theorems~\ref{qn} and~\ref{autcayley} in the next section. We conclude with some remarks and some open questions.

%
%
%
%
%
%
%
%


\section{Preliminaries}\label{sec2}
 We begin with a few definitions from finite geometry. For more details, one may see \cite{strome-sziklai,voorde}.
 By $PG(n, q)$ we mean the Desarguesian projective space obtained from the affine space $AG(n+1, q)$.
\begin{defn}\label{defcone}{\rm
A cone with vertex $A \subset PG(k,q)$ and base $B \subset PG(n-k-1, q)$, where $PG(k,q)\cap PG(n-k-1, q) = \emptyset$, is the set of points lying on the lines connecting points of $A$ and $B$.}
\end{defn}
\begin{defn}{\rm
Let $V$ be an $(n+1)$-dimensional vector space over a finite field $\bF$. A subset $S$ of $PG(V)$ is called an $\bF_q$-linear set if there exists a subset $U$ of $V$ that forms an $\bF_q$-vector space, for some $\bF_q \subset \bF$, such that $ S = \mathcal{B}(U)$, where $$\mathcal{B}(U) := \{\langle u \rangle_{\bF} : u \in U \setminus \{\mzero\} \} $$ and where $\langle u \rangle_{\bF}$ denotes the projective point of $PG(V),$ corresponding to the vector $u$ of $U \subset V$. }
\end{defn}
Further details about $\bF_q$-linear sets can be found in~\cite{voorde}, for instance.

The projective space $PG(n,q)$ can be partitioned into an affine space $AG(n, q)$ and a hyperplane at infinity, denoted by $H_\infty$. 
\begin{defn}{\rm Following~\cite{strome-sziklai}, we say that a set of points $U \subset AG(n, q)$ determines the direction $d \in H_\infty,$ if there is an affine line through $d$ meeting $U$ in at least two points. }\end{defn}
We now state the main theorem of \cite{strome-sziklai} which will be relevant in our setting.

\begin{thm}\label{cone}
Let $U \subset AG(n,\bF_q), n \geq 3, |U|=q^k$. Suppose that $U$ determines at most $\frac{q+3}{2}q^{k-1}+q^{k-2}+\dots+q^2+q$ directions and suppose that $U$ is an $\bF_p$-linear set of points, where $q = p^h,$ $ p>3$ prime. If $n-1 \geq (n-k)h,$ then $U$ is a cone with an $(n-1-h(n-k))$-dimensional vertex at $H_{\infty}$ and with base a $\bF_q$-linear point set $U_{(n-k)h}$ of size $q^{(n-k)(h-1)}$, contained in some affine $(n-k)h$-dimensional subspace of $AG(n,q)$. 
\end{thm}
We end this section by recalling another result that appears in \cite{dobson} as Theorem~4.2. 
\begin{thm}\label{thm:A} Let $G$ be a permutation group on $\Omega$ with a proper self-normalizing abelian regular subgroup. Then $|\Omega|$ is not a prime power.\end{thm}

%

\section{Proofs of the Theorems}\label{sec3}


In this section we prove Theorems~\ref{qn} and~\ref{autcayley} starting with the proof of Theorem~\ref{autcayley}.  We believe that this result is only the tip of an iceberg: its current statement  has been tailored to the context of our setting, and uses some ideas that appear in~\cite[Section 3]{dobson} and~\cite{JS}. 


\begin{proof}[Proof of Theorem~$\ref{autcayley}$]
	We suppose that ${\bf (i)}$ does not hold, that is, $K$ is a proper subgroup of $\mathrm{Aut}(\mathrm{Cay}(G,S))$; we show that ${\bf (ii)}$ holds. Write $\Gamma:=\mathrm{Cay}(G,S)$.
	
	Let $B$ be a subgroup of $\mathrm{Aut}(\Gamma)$ with $K<B$ and with $K$ maximal in $B$. Suppose that $K\lhd B$. As $G$ is characteristic in $K$, we get $G\lhd B$. In particular, every element $\varphi$ in $B\setminus K$ satisfies~{\bf (ii)}. 
	
	Suppose then that $K$ is not normal in $B$. Since $K$ is maximal in $B$ and $G\lhd K$, we have $\nor B G=K $.  Suppose that there exists $b\in B\setminus K$ such that $L:=\langle G,G^b\rangle$ (the smallest subgroup of $B$ containing $G$ and $G^b$) satisfies $L\cap K=G$. We claim that we are now in the position to apply \cite[Theorem 4.2]{dobson} (and implicitly some ideas from \cite{JS}). Indeed, as $\nor L G=\nor B G\cap L=K\cap L=G$, $L$ is a transitive permutation group on the vertices of $\Gamma$ with a proper regular self-normalizing abelian subgroup $G$. (Observe that $G$ is a proper subgroup of $L$ because $b\notin\nor B G=K$.) From~\cite[Theorem 4.2]{dobson}, $|G|$ is not a prime power, which is a contradiction because $|G|=q^n$, see also Theorem~\ref{thm:A}. This proves that, for every $b\in B\setminus K$, we have $\langle G,G^b\rangle\cap K>G$.
	
	Fix $b\in B\setminus K$. Now, $G$ and $G^b$ are abelian and hence $G\cap G^b$ is centralized by $\langle G,G^b\rangle$. From the preceding paragraph, there exists $k\in \langle G,G^b\rangle\cap K$ with $k\notin G$. Observe now that $K=\mathbb{F}_q^n\rtimes \mathbb{F}_q^*$ is a Frobenius group with kernel $G=\mathbb{F}_q^n$ and complement $\mathbb{F}_q^*$. Therefore, $k$ acts by conjugation fixed-point-freely on $G\setminus\{\mzero\}$. As $k$ centralizes $G\cap G^b$, we deduce $|G\cap G^b|=1$.
	
	Let $C:=\bigcap_{x\in B}K^x$ be the core of $K$ in  $B$. As $G\cap G^b=1$,  $K\cap K^b$ has no non-identity $q$-elements. Therefore $C\cap G=1$. As $C\lhd B$ and $C\le K$,  $C$ is a normal subgroup of the Frobenius group $K$ intersecting its kernel on the identity. This yields $C=1$.
	
	Let $\Omega$ be the set of right cosets of $K$ in $B$. From the paragraph above, $B$ acts faithfully on $\Omega$. Moreover, as $K$ is maximal in $B$, the action of $B$ on $\Omega$ is primitive. Therefore $B$ is a finite primitive group with a solvable point stabilizer $K$. In \cite{LZ}, Li and Zhang have explicitly determined such primitive groups: these are classified in \cite[Theorem 1.1]{LZ} and \cite[Tables I--VII]{LZ}. Now, using the terminology in \cite{LZ}, a careful (but not very difficult) case-by-case analysis on the tables in \cite{LZ} shows that $B$ is a primitive group of affine type, that is, $B$ contains an elementary abelian  normal $r$-subgroup $V$, for some prime $r$. For this analysis it is important to keep in mind that the stabilizer $K$ is a Frobenius group with kernel the elementary abelian group $G\cong \mathbb{F}_q^n$ and $n>1$. 
	
	Let $|V|=r^t$. Now, the action of $B$ on $\Omega$ is permutation equivalent to the natural action of $B=V\rtimes K$ on $V$, with $V$ acting via its regular representation and with $K$ acting by conjugation. Observe that $q\ne r$, because $K$ acts faithfully and irreducibly as a linear group on $V$ and hence $K$ contains no non-identity normal $r$-subgroups. Observe further that $|B|=|V||K|=r^t\cdot q^n\cdot (q-1)$.

	We are finally ready to reach a contradiction and to do so, we go back studying the action of $B$ on the vertices of $\Gamma$. Observe that $B$ is solvable because $V$ is solvable and so is $B/V\cong K$. We write $B_\mzero$ for the stabilizer in $B$ of the vertex $\mzero$ of $\Gamma$. As $G$ acts regularly on the vertices of $\Gamma$, we obtain $B=B_\mzero G$ and $B_\mzero \cap G=1$. In particular, $|B_\mzero|=r^t\cdot (q-1)$. Observe that $B_\mzero$ is a Hall $\Pi$-subgroup of the solvable group $B$, where $\Pi$ is the set of all the prime divisors of $q-1$ together with the prime $r$. As $V$ is a $\Pi$-subgroup, from the theory of Hall subgroups (see for instance \cite{doerk}, Theorem 3.3), $V$ has a conjugate contained in $B_\mzero$. Since $V\lhd B$, we have $V\le B_\mzero$. This is clearly a contradiction because $V$ is normal in $B$, but $B_\mzero$ is core-free in $B$ being the stabilizer of a point in a transitive permutation group.
\end{proof}
For the next lemma, recall that $\Hy_0:=\{(x_1,\ldots,x_{n-1},0):x_i\in\bF_q\textrm{\ for\ each\ } i\in\{1,\ldots, n-1\}\}$. In what follows, $G_{n,S}$ will denote the Cayley graph $\mathrm{Cay}(\bF_q^n,S)$ and $S=\widetilde{S}\setminus\{\mzero\}$ for some set $\widetilde{S}=\displaystyle\bigcup_{\ell\in\List}\ell$, where $\List$ is a collection of lines in $\bF_q^n$ with each $\ell\in\List$ satisfying $\ell\cap\Hy_0=\{\mzero\}$. 
\begin{lem}\label{coloring}  $\chi(G_{n,S}) = q$.
\end{lem} 
\begin{proof}
%
%

Observe that each line that belongs to the set $S$ gives rise to a clique of size $q$ in the graph $G_{n,S}$. Therefore $\chi(G_{n,S}) \ge q$. On the other hand, for a fixed $v \in S$, the partition $(C_\lambda)_{\lambda\in\mathbb{F}_q}$, where $C_\lambda:=\{w + \lambda v : w \in \mathcal{H}_0 \}$, of the vertex set $\bF_q^n$ is a proper coloring of the graph $G_{n,S}$. Indeed,  for any $u,v \in C_\lambda$, we have $u-v = w \notin S$, so the sets $C_{\lambda}$ are independent in $G_{n,S}$ for each $\lambda\in\mathbb{F}_q$.
\end{proof}

%
%
%


%

\begin{lem}\label{hsp}	Assume that $q$ is prime. Let $\widetilde{S}$ be the random set corresponding to a union of lines $\ell$ in $\mathbb{F}_q^n$ with $\ell\cap \mathcal{H}_0=\{\mzero\}$ and where each $\ell\in\bF_q^n$ is chosen independently with probability $\frac{1}{2}$; and let $S=\widetilde{S}\setminus\{\mzero\}$. Then \begin{equation*}\bP\left(\chi_D(G_{n,S}) > q\right)\geq 1- \exp\left(-\frac{q^{n-3}}{4}\right).\end{equation*}	
\end{lem}
\begin{proof} First, note that $\bE(|S|) = \frac{q^{n-1}}{2}$, so taking $\delta =\frac{1}{q}$ and $\mu = \bE(|S|)$ in the Chernoff bound (see $(2.6)$ on page $26$ of \cite{janson}) we obtain
	$$\bP\left(|S| < \frac{q^{n-1}-q^{n-2}}{2}\right) \leq \exp\left(-\frac{q^{n-3}}{4}\right).$$ In particular, with probability at least $1-\exp(-q^{n-3}/4)$, we have $|S|>\frac{q^{n-1}-q^{n-2}}{2}$. We may thus  assume $|S| > \frac{q^{n-1}-q^{n-2}}{2}$ in what follows. 

 We claim that every color class in a proper $q$-coloring of $G_{n,S}$ is an affine hyperplane of $\bF_q^n$. To see why, let $C_1,\ldots,C_q$ be independent sets in $G_{n,S}$ witnessing a proper $q$-coloring of $G_{n,S}$.  Fix $v \in S$ and consider the line $\ell_v := \{ \lambda v : \lambda \in \bF_q\}$ along with its translates $\ell_v + w:=\{\lambda v +w:\lambda\in\bF_q\}$, for $w\in \mathcal{H}_0$. Each set $\ell_v+w$ is a clique of size $q$ in $G_{n,S}$, and these cliques partition the vertex set of $G_{n,S}$, so in particular each $C_i$ contains at most one vertex from each of these translates $\ell_v+w$. Consequently,  $|C_i| \leq q^{n-1}$ for all $i\in\{1,\ldots, q\}$. By size considerations, it follows that $|C_i| = q^{n-1}$ for each $i\in\{1,\ldots,q\}$.

Consider a color class $C$.  Suppose $C$ determines at least $\frac{q+3}{2}q^{n-2}+q^{n-3}+\dots+q^2+q+1$ directions. Then if $\langle C \rangle$ denotes the set of all affine lines intersecting at least two points in $C$, we have $|\langle C\rangle|+|S|>1+q+\cdots+q^{n-1}$, so $\langle C \rangle \cap S \neq \emptyset$. However, this contradicts the assumption that $C$ is an independent set in $G_{n,S}$. Therefore $C$ determines at most $\frac{q+3}{2}q^{n-2}+q^{n-3}+\dots+q^2+q$ directions. Since $q$ is prime, by Corollary 10 in~\cite{strome-sziklai}, it follows that $C$ is an $\bF_q$-linear set. Hence, by Theorem~\ref{cone}, the color class $C$ is a cone with an $n-2$ (projective) dimensional vertex $\V$ at $H_\infty$ and an affine point $u_1$ as base. In particular, the affine plane corresponding to the $\bF_q$-subspace spanned by $\V$ passing through the affine point $u_1$  is contained in $C$. Since  $|C| = q^{n-1}$, it follows that $C$ is this affine hyperplane, and this proves the claim.

	To complete the proof, observe that for each $\lambda \in \bF_q^*\setminus\{1\}$, the map $\varphi_\lambda(x) = \lambda x $, $x \in \bF_q^n$ fixes each  color class. Moreover, $\varphi_\lambda$ fixes the set $S$ and $\varphi_\lambda(u) - \varphi_\lambda(v) = \varphi_\lambda(u-v)$, so $\varphi_\lambda$ is a non-trivial automorphism which fixes each color class. Therefore $\chi_D(G_{n,S})>q$.
\end{proof}

\begin{lem}\label{aut}
If $n\ge 5$ and $q\ge 5$ is prime, then $\mathrm{Aut}(G_{n,S}) \cong \bF_q^n \rtimes \bF_q^*$ with probability at least $1-2^{(-\frac{q^{n-1}}{3})}$. \end{lem}
%
\begin{proof}
	Since $G_{n,S}$ is a Cayley graph on the additive group $G=\bF_q^n$,  by Theorem~\ref{autcayley}, either $\mathrm{Aut}(G_{n,S})=K \cong \bF_q^n \rtimes \bF_q^*$ or there exists $\varphi \in \mathrm{Aut}(G_{n,S})\setminus K$ with $\varphi$ normalizing $G=\bF_q^n$.  We show that with probability at least $1-2^{(-\frac{q^{n-1}}{3})}$, there is no $\varphi$ satisfying the latter condition. 
%

 Suppose $\varphi\in\Aut(G_{n,S})$ normalizes $\bF_q^n$. If $a=\varphi(\mzero)$ and $\lambda_a:\bF_q^n\to \bF_q^n$ is the right translation via $a$, then $\lambda_{a}^{-1}\varphi$ is an automorphism of $G_{n,S}$ normalizing $\bF_q^n$ and with $(\lambda_a^{-1}\varphi)(\mzero)=(\lambda_a^{-1})(\varphi(\mzero))=(\lambda_a^{-1})(a)=a-a=\mzero$. Therefore, without loss of generality, we may assume that $\varphi(\mzero)=\mzero$. Since $S$ is the neighbourhood of $\mzero$ in $G_{n,S}$, we get $\varphi(S)=S$. Moreveor, since $\varphi$ acts as a group automorphism on $\bF_q^n$, we have $\varphi\in \mathrm{GL}_n(q)$.
 
Now, for $\varphi\in \mathrm{GL}_n(q)$, let $E_{\varphi}$ denote the event  $\varphi(S)=S$. Let $\mathcal{L}$ denote the set of all lines $\ell$ with $\ell\cap\Hy_0=\emptyset$. Also, let $\mathrm{Orb}_{\varphi}(\ell) = \{\ell, \varphi(\ell), \varphi^2(\ell), \dots, \varphi^k(\ell)\}$ where $\varphi^{k+1}(\ell)=\ell$. Then $$\bP(E_{\varphi})\le \prod_{i=1}^{N_\varphi}2^ {1-|\mathrm{Orb}_{\varphi}(\ell_i)|} = 2^ {N_\varphi-|\mathcal{L}|},$$ where
$N_{\varphi}$ denotes the number of distinct orbits of $\varphi$ in $\mathcal{L}$. Setting  $\G= GL(n,q)\setminus\{\lambda I:\lambda\in\bF_q^*\}$, we have 
\begin{equation}\label{prob} \bP\left(\bigcup_{\varphi\in \G}E_{\varphi}\right)\le \sum_{\varphi\in\G}\bP(E_{\varphi})\le 2^{-|\La|}\sum_{\varphi\in\G} 2^{N_{\varphi}}.
\end{equation}
Let  $F_\varphi:= |\{\ell\in\La: \varphi(\ell) =\ell\}|$ and $F:= \mathop{\max}\limits_{\varphi \in \mathcal{G}}{F_\varphi}$. Now $N_\varphi \le F + \frac{|\La| - F}{2} = \frac{F+|\La|}{2}$. Thus, it suffices to give a suitable upper bound for $F$. Towards that end, we note that, if  $F_{\varphi}=F$ for $\varphi\in\G$, then every line $\ell$ fixed by $\varphi$ corresponds to an eigenvector of $\varphi$. If $\E_1,\E_2\ldots,\E_k$ denote the eigenspaces of $\varphi$ for some distinct eigenvalues $\lambda_1,\ldots,\lambda_k$, then 
\begin{displaymath} F_{\varphi}\le \sum_{i=1}^k \left(\binom{\dim \E_i}{1}_q-\binom{\dim (\E_i\cap\Hy_0)}{1}_q\right)\le q^{n-2}+1.\end{displaymath}
Similarly, we have $|\La|=  \binom{n}{1}_q-\binom{n-1}{1}_q= q^{n-1}$, and so by~\eqref{prob}, we have
$$ \bP\left(\bigcup_{\varphi\in \G}E_{\varphi}\right)\le |\G|2^{\frac{F-|\La|}{2}}<q^{n^2}2^{-(\frac{q^{n-1}-q^{n-2}-1}{2})}<2^{-(\frac{q^{n-1}}{3})},$$
 for $q \ge 5$, $n\ge 6$.
\end{proof}

Computations and estimates similar to the ones presented in the proof of Lemma~\ref{aut} have been proved useful in a variety of problems, see for instance{~\cite{BP}},{~\cite{GuSp} and~\cite[Section~$6.4$]{PSV1}.

\begin{proof}[Proof of Theorem~$\ref{qn}$]	
Given $k \in \bN$ with $k\ge 4$, pick a prime number $q$ with $k<q<2k$. Consider the random graph $G_{n,S}$ of the group $\bF_q^n$ as constructed above. By Lemmas~\ref{hsp} and~\ref{aut}, with positive probability, the graph $G_{n,S}$ satisfies the statements of both lemmas, and hence satisfies the conclusions of Theorem~\ref{qn}.\end{proof}

\section{Concluding Remarks}\label{sec4}
\begin{itemize}
\item  We observe that, for $S$ chosen randomly as in the proof of our result, the distinguishing chromatic number of $G_{n,S}$  is $q+1$ with high probability. Indeed, consider the $q$-coloring $C$ described in Lemma~\ref{coloring}. Re-color the vertex $\mzero$ using an additional color. Then the coloring described by the partition $C' = C \cup \{\mzero\}$ is a proper, distinguishing coloring of $G_{n,S}$ with $q+1$ colors. In fact, $C'$ is clearly proper, and to show that it is distinguishing, consider $\varphi \in \Aut(G_{n,S})=\bF_q^n \rtimes \bF_q^*$ (by Lemma \ref{aut}) that fixes every color class.  
 Write $\varphi(x) = \lambda x + b$ with $\lambda \in \bF_q^*, b \in \bF_q^n$. Since $\varphi$ fixes the color class containing $\mzero$, we have $b = \mzero$. Also, $x$ and $\lambda x$ cannot be in same color class unless $\lambda =1$. Therefore $\varphi$ is the identity automorphism.    \\
 It is  interesting to determine if one can obtain families of vertex-transitive graphs with $\chi_D(G)>\chi(G)+1$, with `small' automorphism groups and with $\chi(G)$ being arbitrarily large. In fact, for $k\in\bN$, there is no known family of vertex-transitive graphs for which $\chi_D(G)>\chi(G)+1>k$ and $|\Aut(G)|=O(|V(G)|^{O(1)})$. It is plausible that Cayley graphs over certain groups may provide the correct constructions.
 \item Theorem~\ref{qn} establishes,  for any fixed $k$, the existence of vertex-transitive graphs $G_n=(V_n,E_n)$ with $\chi_D(G_n)>\chi(G_n)>k$ and with $|\Aut(G_n)|<2k|V_n|$. It would be interesting to obtain a similar family of graphs that satisfy  with $\chi_D(G_n)>\chi(G_n)>k$ and with $|\Aut(G_n)|\le C|V_n|$, for some absolute constant $C$.

\end{itemize}

\textbf{Acknowledgments}

The first and second authors would like to thank Ted Dobson for useful discussions.


\bibliographystyle{abbrv}
\bibliography{refer}

\begin{thebibliography}{10}

\bibitem{BP}
N.~Balachandran and S.~Padinhatteeri.
\newblock {$\chi_D(G)$, $|Aut(G)|$} and a variant of the motion lemma.
\newblock {\em Ars Math. Contemp.}, 12(1), 2016.

\bibitem{kneser}
Z.~Che and K.~L. Collins.
\newblock {The Distinguishing Chromatic Number of Kneser Graphs}.
\newblock {\em Electron. J. Combin.}, 20(1), 2013.

\bibitem{Hemanshu}
J.~Choi, S.~G. Hartke, and H.~Kaul.
\newblock Distinguishing chromatic number of cartesian products of graphs.
\newblock {\em SIAM J. Discrete Math.}, 24(1):82--100, 2010.

\bibitem{Collins2009}
K.~L. Collins, M.~Hovey, and A.~N. Trenk.
\newblock Bounds on the distinguishing chromatic number.
\newblock {\em Electron. J. Combin}, 16, 2009.

\bibitem{collins}
K.~L. Collins and A.~N. Trenk.
\newblock The distinguishing chromatic number.
\newblock {\em Electron. J. Combin}, 13, 2006.

\bibitem{dobson}
E.~Dobson, P.~Spiga, and G.~Verret.
\newblock Cayley graphs on abelian groups.
\newblock {\em Combinatorica}, 4:1--23, 2015.

\bibitem{doerk}
K.~Doerk and T.~O. Hawkes.
\newblock {\em Finite Soluble Groups}.
\newblock De Gruyter Expositions in Mathematics 4, 1992.

\bibitem{GuSp}
S.~Guest and P.~Spiga.
\newblock Finite primitive groups and regular orbits of group elements.
\newblock {\em Trans. Amer. Math. Soc.}, 369(2):997--1024, 2017.

\bibitem{JS}
E.~Jabara and P.~Spiga.
\newblock Abelian carter subgroups in finite permutation groups.
\newblock {\em Arch. Math. (Basel)}, 101:301--307, 2013.

\bibitem{janson}
S.~Janson, T.~{\L}uczak, and A.~Ruci{\'n}ski.
\newblock {\em Random Graphs}.
\newblock John Wiley \& Sons, Inc.,New York, 2000.

\bibitem{LZ}
C.~H. Li and H.~Zhang.
\newblock The finite primitive groups with soluble stabilizers, and the
  edge-primitive $s$-arc transitive graphs.
\newblock {\em Proc. Lond. Math.Soc.}, 103:441--472, 2011.

\bibitem{PSV1}
P.~Poto\v{c}nik, P.~Spiga, and G.~Verret.
\newblock Asymptotic enumeration of vertex-transitive graphs of fixed valency.
\newblock {\em J. Combin. Theory Ser. B}, 122:221--240, 2017.

\bibitem{strome-sziklai}
L.~Strome and P.~Sziklai.
\newblock {Linear point sets and R\'{e}dei type {$k$}-blocking sets
  {$PG(n,q)$}}.
\newblock {\em J. Algebraic Combin.}, 14:221--228, 2001.

\bibitem{voorde}
G.~V. Voorde.
\newblock {\em Blocking sets in finite projective spaces and coding theory}.
\newblock PhD thesis, Ghent University, 2010.

\end{thebibliography}

\end{document}